\documentclass{article}

\usepackage{fullpage}

\usepackage{amsmath}
\usepackage{amssymb}
\usepackage{latexsym}
\usepackage{amsthm}
\usepackage[retainorgcmds]{IEEEtrantools}

\usepackage{authblk}

\usepackage[capitalize]{cleveref}

\usepackage{tikz}
\usetikzlibrary{arrows,calc}

\tikzstyle{arc}=[->,shorten <=3pt, shorten >=3pt,
                 >=stealth, line width=1.1pt]
\tikzstyle{edge}=[shorten <=2pt, shorten >=2pt,
                  >=stealth]
\tikzstyle{vertex}=[circle, fill=white, draw,
                    minimum size=5pt,
                    inner sep=0pt, outer sep=0pt, fill = black]
\DeclareMathOperator{\girth}{girth}

\newtheorem{theorem}{Theorem}
\newtheorem*{theorem*}{Theorem}
\newtheorem{lemma}[theorem]{Lemma}
\newtheorem{observation}[theorem]{Observation}
\newtheorem{proposition}[theorem]{Proposition}
\newtheorem{corollary}[theorem]{Corollary}
\newtheorem*{problem}{Problem}
\newtheorem*{question}{Question}
\newtheorem{conjecture}[theorem]{Conjecture}

\usepackage{minitoc}

\title{Double circuits in bicircular matroids%
\thanks{The content of this note is included and extended in: arXiv:2209.06591}}

\author[1,2]{Santiago~Guzm\'an-Pro\thanks{sanguzpro@ciencias.unam.mx}}

\author[2]{Winfried~Hochst\"attler\thanks{winfried.hochstaettler@fernuni-hagen.de}}

\affil[1]{\small{Facultad de Ciencias\\
Universidad Nacional Aut\'onoma de M\'exico\\
C.P. 04510, Ciudad Universitaria, M\'exico}}

\affil[2]{FernUniversit\"at in Hagen\\
Fakult\"at f\"ur Mathematik und Informatik\\
 58084 Hagen}

\begin{document}
\date{}
\maketitle

\begin{abstract}
The first non-trivial case
of Hadwiger's conjecture for oriented matroids reads as follows. If $\mathcal{O}$
is an $M(K_4)$-minor free
oriented matroid, then 
$\mathcal{O}$ has a now-where $3$-coflow, i.e., it is $3$-colourable in the sense of
Hochst\"attler-Ne\v{s}et\v{r}il.
The class of gammoids is a class of $M(K_4)$-free orientable
matroids and it is the minimal minor-closed class that contains all transversal
matroids. Towards proving the previous statement for the class of gammoids,
Goddyn, Hochst\"attler, and Neudauer conjectured that every gammoid has a
positive coline (equivalently, a positive double circuit), which implies that all
orientations of gammoids are $3$-colourable. This conjecture stems from their
proof that every cobicircular matroid has a positive double circuit. In this brief
note we disprove Goddyn, Hochst\"attler, and Neudauers' conjecture by
exhibiting a large class of bicircular matroids that do not contain positive
double circuits.
\end{abstract}

\section{Introduction}

Hadwiger's Conjecture is a well-known and long open conjecture regarding proper graph
colourings. It states that for every positive integer $k$ if a graph $G$ contains no
$K_{k+1}$-minor, then $G$ is $k$-colourable. This conjecture has been proven 
true for $k\le 5$ \cite{robsertsonC13}, and remains open for larger integers. 

The notion of proper graph colourings is extended to oriented matroids in
different ways.
In particular, Hochst\"attler and Ne\v{s}et\v{r}il~\cite{hochstattlerEJC27} propose to
define the chromatic number of an oriented matroid in terms of nowhere-zero
coflows ($NZ$ coflows). It turns out that Hadwiger?s
conjecture can be generalized to this context; but in this scenario,
the first non-trivial case remains open and it reads as follows.

\begin{conjecture}\label{conj:hadw}
Every (loopless) $M(K_4)$-minor free oriented matroid has a nowhere-zero $3$-coflow.
\end{conjecture}

Towards proving Conjecture~\ref{conj:hadw}, Goddyn, Hochst\"attler and Neudauer,
introduce the class of \textit{Generalized Series Parallel} ($GSP$) oriented matroids
and show that every $GSP$ oriented matroid has a NZ
$3$-coflow~\cite{goddynDM339}. 
It might be too much to hope for, but if every $M(K_4)$-free oriented matroid
is $GSP$, then Conjecture~\ref{conj:hadw} follows directly. In any case,
this raises the fundamental problem of determining when a class $\mathcal{C}$
of oriented matroids is a subclass of $GSP$ oriented matroids. To this end
and in the same work, the previously mentioned authors show that if
$\mathcal{C}'$ is a class of orientable matroids closed under minors such that
every member of $\mathcal{C}'$ has a  positive coline, then the class $\mathcal{C}$
of all orientations of matroids in $\mathcal{C}'$ is a class of $GSP$ oriented
matroids. Finally, they show that every bicircular matroid has a positive coline,
so, if $\mathcal{O}$ is an oriented bicircular matroid,  then
$\mathcal{O}$ is $GSP$ and thus it has a NZ $3$-coflow. 

Bicircular matroids are transversal matroids, and the smallest class
closed under minors that contains transversal matroids is the class of gammoids. 
In turn, the class gammoids is a class of $M(K_4)$-free orientable matroids, so 
Goddyn, Hochst\"attler and Neudauer pose the following conjecture. 

\begin{conjecture}\label{conj:coline}\cite{goddynDM339}
Every simple gammoid of rank at least two has a positive coline.
\end{conjecture}

In this work, we disprove this conjecture by exhibiting a large class of
cobicircular matroids that do not have positive colines.

The rest of this work is organized as follows. In Section~\ref{sec:prelim},
we introduce all concepts needed to state  Conjecture~\ref{conj:coline}.
In Section~\ref{sec:counter}, we introduce bicircular
matroids and prove the necessary results to exhibit a class of
counterexamples to the previously mentioned conjecture. 
Finally, in Section~\ref{sec:conclusions} we notice that a simple
observation made in Section~\ref{sec:counter} relates 
to known results about representability of bicircular matroids, and
we pose some questions that arise from this relation.

\section{Preliminaries}
\label{sec:prelim}

As previously mentioned, Conjecture~\ref{conj:coline}
is motivated by showing that all orientations of gammoids are $GSP$, 
and thus exhibiting a large class of $M(K_4)$-free oriented matroids
that admit a $NZ$ $3$-coflow. Nonetheless, we can state and disprove
Conjecture~\ref{conj:coline} without introducing $GSP$ oriented matroids
and $NZ$ $3$-coflows. For the sake of simplicity, 
we only introduce concepts and definition used in this work. 
A standard reference for matroid theory is \cite{oxley1992}.

Let $M$ be a matroid. A \textit{copoint} of $M$ is a hyperplane, that is, 
a flat of codimension $1$. A \textit{coline} of $M$ is a flat of codimension $2$. 
If a coline $L$ is contained in a copoint $H$, we say that $H$ is a \textit{copoint on}
$L$. It is not hard to notice that if $L$ is a coline of a matroid $M$ then there
is a partition $(H_1,\dots, H_k)$ of $E(M)\setminus  L$ such that the copoints
on $L$ are $L\cup H_i$ for $i\in\{1,\dots k\}$. Furthermore, this partition is unique
(up to permutations) so we call it the \textit{copoint partition} of $L$, and
define the \textit{degree} of $L$ to be $k$. A class $H_i$ is \textit{singular}
if $|H_i| = 1$; otherwise it is a \textit{multiple} class.
A coline $L$ is \textit{positive} if there are more
singular that multiple classes  in its copoint partition.
It turns out that positive colines and $GSP$ oriented matroids are related as follows.

\begingroup
\def\thetheorem{\cite{goddynDM339}}
\begin{proposition}
Let $\mathcal{C}$ be minor closed class of orientable matroids. 
If every simple matroid in $\mathcal{C}$ has a positive coline, then
every orientation of a matroid in $\mathcal{C}$ is $GSP$.
\end{proposition}
\addtocounter{theorem}{-1}
\endgroup

Using this proposition, Goddyn, Hochst\"attler, and Neudauer
\cite{goddynDM339} show that every orientation of a bicircular matroid
is $GSP$. Every bicircular matroid is a transversal matroid, and the class
of gammoids is the smallest dually and minor closed class that contains transversal 
matroids \cite{ingletonJCTB15}. These facts motivate Goddyn, Hochst\"attler,
and Neudauer to conjecture that every orientation of a gammoid is $GSP$. 
Furthermore, they conjecture that the following statement is true.

\begingroup
\def\thetheorem{\ref{conj:coline}}
\begin{conjecture}\cite{goddynDM339}
Every simple gammoid of rank at least two  has a positive coline.
\end{conjecture}
\addtocounter{theorem}{-1}
\endgroup

Circuits are the dual complements of hyperplanes. That is, if $H$ is a
hyperplane of a matroid $M$ then $E(M)\setminus  H$ is a circuit of $M^\ast$. 
A \textit{double circuit} of a matroid $M$ is a set $D$ such that $r(D) = |D|-2$
and for every element $d\in D$ the rank of $D-d$ does not decrease, i.e.\
$r(D-d) = |D| -2 = r(D)$. Dress and Lov\'asz \cite{dressC7} show that if $D$ is a
double circuit then $D$ has a partition $(D_1,\dots, D_k)$ such that the circuits
of $D$ are $D\setminus  D_i$ for $i\in\{1,\dots, k\}$. We call this partition the \textit{circuit
partition} of $D$ and say that the \textit{degree} of $D$ is $k$. 

\begin{observation}\label{obs:dcseries}
Let $M$ be a matroid and  $D\subseteq E(M)$ a double circuit of $M$.
If $D$ is a degree $k$ double circuit, then $M[D]$ is a series 
extension of $U_{k-2,k}$.
\end{observation}
\begin{proof}
With out loss of generality suppose that $D = E(M)$, and let
$(D_1,\dots, D_k)$ be the circuit partition of $D$.  If $|D_i| = 1$ for every
$i\in\{1,\dots, k\}$, then $M \cong U_{k-2,k}$. Now, the claim follows 
by a straightforward induction over the difference  $|D| - k$.
\end{proof}

Similar to how circuits are the dual complements of copoints, 
double circuits are the complements of colines. Moreover, the 
copoint partitions and circuit partitions relate as follows.

\begin{observation}\label{obs:coline-doublec}
Let $M$ be a matroid, $L\subseteq E(M)$ and $(H_1,\dots H_k)$ a partition of
$E(M)\setminus  L$. Then, $L$ is  a coline of $M$ with copoint partition
$(H_1,\dots, H_k)$ if and only if $E(M)\setminus  L$ is a double circuit of $M^\ast$
with circuit partition $(H_1,\dots, H_k)$.
\end{observation}

A \textit{positive double circuit} $D$ is a double circuit with more singular
than multiple classes in its circuit partition. By Observation~\ref{obs:coline-doublec},
a matroid $M$ has a positive double circuit if and only if $M^\ast$
has a positive coline. Since gammoids are closed under duality, 
the following conjecture is equivalent to Conjecture~\ref{conj:coline}.

\begin{conjecture}\cite{goddynDM339}\label{con:posdouble}
Every cosimple gammoid of corank at least two  has a positive double circuit.
\end{conjecture}

\section{Bicircular matroids and double circuits}
\label{sec:counter}

Every bicircular matroid is a transversal matroid~\cite{matthewsQJMOS28},
and so, every bicircular matroid is a gammoid. In this section,  we disprove 
Conjecture~\ref{conj:coline} by showing that its dual statement,
Conjecture~\ref{con:posdouble},
does not hold for bicircular matroids. To do so, we begin by briefly introducing
the class of bicircular matroids.

A standard reference for graph theory is \cite{bondy2008}. In particular, 
given a graph $G$ and a subset of edges $I$  we denote by
$G[I]$ the subgraph of $G$ induced by $I$. That is,  $G[I]$ is the subgraph
of $G$ with edge set $I$ and no isolated vertices.

Let $G$ be a (not necessarily simple) graph with vertex set $V$ and edge set $E$.
The \textit{bicircular matroid} of $G$ is the matroid $B(G)$ with base set $E$
whose independent sets are the edge sets $I\subseteq E$ such that  $G[I]$
contains at most one cycle in every connected component. 
Equivalently, the circuits of $B(G)$ are the edge sets of subgraphs which are
subdivisions of one of the graphs: two loops on the same vertex, two loops joined
by an edge, or three parallel edges joining a pair of vertices.

Mathews \cite{matthewsQJMOS28} noticed that there are only a few uniform
bicircular matroids. 

\begin{theorem}\label{thm:uniformbicircular}\cite{matthewsQJMOS28}
The uniform bicircular matroids are precisely the following:
\begin{itemize}
	\item $U_{1,n}$, $U_{2,n}$, $U_{n,n,}$ ($n\ge 0$);
	\item $U_{n-1,n}$ ($n\ge 1$);
	\item $U_{3,5}$, $U_{3,6}$ and $U_{4,6}$.
\end{itemize}
\end{theorem}

Recall that if a matroid $M$ has a double circuit of degree $k$ then 
$M$ contains a $U_{k-2,k}$ minor (Observation~\ref{obs:dcseries}).

\begin{corollary}\label{cor:6}
The  degree of a double circuit in a bicircular matroid is at most $6$.
\end{corollary}

Suppose that a graph $G$ is obtained by subdividing edges of a 
graph $H_G$ with minimum degree $3$. It is not hard to notice that
$H_G$ is unique up to isomorphism. The \textit{subdivision classes} of
$G$ are the sets of edges that correspond to a series of subdivisions of
an edge in $H_G$. An \textit{unsubdivided edge} of $G$ is an edge of $G$
that is an edge of $H_G$. Clearly, if $G$ has no
leaves an edge $xy$ is an unsubdivided edge of $G$
if and only if  $d_G(x),d_G(y)\ge 3$.

\begin{lemma}\label{lem:basic}
Let $G$ be a graph and $D\subseteq E$ a double circuit of $B(G)$.
Then $G[D]$ has no leaves and contains at most $4$ vertices
$x_1, x_2,x _3$ and $x_4$ of degree greater than or equal to $3$. Moreover,
every subdivision class of $G[D]$ belongs to the same class of the circuit partition
of $D$.
\end{lemma}
\begin{proof}
Since every element of $D$ belongs to a circuit of $D$, then $D$ has no
coloops so $G[D]$ has no leaves.
Let $V'$ be the vertex set of $G[D]$ and $r'$ the rank of $D$, so $r'= |V'|$.
Since $D$ has no leaves then $d(x)\ge 2$ for every $e\in V'$. Let $t$ be
the number of vertices in $V'$ with degree at least $3$ in $G[D]$. 
By the handshaking lemma, $2|D| \ge 3t + 2(r'-t)$. Since  $D$ is a double circuit
then $r' = |D| -2$, and thus  $2(r'+2) \ge 3t + 2(r'-t)$, so $t \le 4$. 
Finally, let $S$ be a subdivision class of $G[D]$. Then $S$ is the edge set of a
path $P$ path such that the internal vertices of $P$ have degree $2$ in $G[D]$.
Thus, every cycle in $G[D]$ that contains an edge of $P$ contains all edges of $P$. 
Hence, every circuit of $B(G)$ in $D$ that contains an edge in $P$ contains
all of them, so $E(P)$ is contained  in some circuit class of $D$.
\end{proof}

Given a double circuit $D$ of a bicircular matroid $B(G)$, a
\textit{distinguished vertex} of $D$ is a vertex of degree at least $3$ in $G[D]$.
Since $G[D]$ has no leaves, the subdivision classes of $G[D]$ correspond to
paths that contain distinguished vertices (only) as endpoints. In particular,
unsubdivided edges of $G[D]$ are edges incident in distinguished vertices.

\begin{theorem}\label{thm:main}
Let $G$ be a graph. If $\girth(G) \ge 5$ then $B(G)$ has no positive double circuits.
\end{theorem}
\begin{proof}
We proceed by contrapositive. Suppose that $D$ is a positive double circuit in $B(G)$
of degree $k$ with circuit partition $(D_1,\dots D_k)$ where $k\le 6$ by
Corollary~\ref{cor:6}.
If $k \in \{1,2,3,4\}$ with out loss of generality assume that $(D_1,\dots, D_{k-1})$
are singular circuit classes of $D$. Since the union $D'$ of these classes is $D\setminus D_k$, 
then this union is a circuit of $D$. Thus, $G[D']$ contains two cycles of $G$, so
 $\girth(G)\le |D'| \le 3$. 

Now suppose that $k \in\{5,6\}$ and $(D_1,\dots, D_k)$ has at least $4$ simple classes
$\{e_1\},\dots \{e_4\}$.  By the moreover  statement of Lemma~\ref{lem:basic}, the
edges $e_1$, $e_2$, $e_3$ and $e_4$ must be unsubdivided edges of $G[D]$. 
Thus, the endpoints of these edges are distinguished vertices of $G[D]$, which
by the same lemma there are at most $4$ of these vertices. Putting all of this 
together we conclude that $D'$ is a set of four edges such that $G[D']$ has
at most $4$ vertices. Therefore, $G[D']$ contains at least one cycle of
$G$, and so $\girth(G)\le |D'| \le 4$.

The only remaining case is when the degree of $D$ is $5$ and it has $3$ singular
classes. In this case there are $3$ unsubdivided edges $e_1,e_2,$ and $e_3$
of $G[D]$. We claim that  $\{e_1,e_2,e_3\}$ contains a cycle of $G$, and thus
$\girth(G)\le 3$. Anticipating a contradiction, suppose that $\{e_1,e_2,e_3\}$ does
not contain a cycle of $G$. This implies that $D$ has four distinguished
$x_1$, $x_2$, $x_3$ and $x_4$. Notice that if we contract a subdivided edge
$e$ of $G[D]$ we obtain
a double circuit $D'$ of $B(G)/e$ with the same degree as $D$. Inductively, we end
up with a double circuit $D_0$ and a graph $H$ with edge set $D_0$ such that 
$\{e_1,e_2,e_3\}\subseteq  D_0$ and $V(H) = \{x_1,x_2,x_3,x_4\}$. Moreover, 
$\{e_1,e_2,e_3\}$ spans a tree of $H$. On the other hand,
each edge $e_i$ for $i\in\{1,2,3\}$ belongs to a singular class of the circuit partition
of $D_0$. Since the rank of $D_0$ is $4$ then
$D_0$ contains at most six edges. Also, the circuit partition of $D_0$ has $5$ classes, 
so there must be an edge  $e_4\in D_0\setminus\{e_1,e_2,e_3\}$ that belongs to 
a singular class. Notice that that $\{e_1,e_2,e_3,e_4\}$ contains at most one
cycle of $H$ since $\{e_1,e_2,e_3\}$ spans a tree of $H$. On the other hand,
$D_0\setminus \{e_1,e_2,e_3,e_4\}$ is a class of $D_0$, so
$\{e_1,e_2,e_3,e_4\}$ is a circuit of $D_0$, i.e.\ $\{e_1,e_2,e_3,e_4\}$ contains
two cycles of $H$. We arrive at this contradiction by assuming that
$\{e_1,e_2,e_3\}$ does not contain a cycle of $G$,  thus $\girth(G)\le 3$, and
the theorem follows.
\end{proof}

This statement yields a large class $\mathcal{C}$ of bicircular matroids with no positive
double circuits.  For instance, if $G$ is the dodecahedron graph and $P$ the
Petersen graph (Figure~\ref{fig:petdod}) then $B(G)$
and $B(P)$ are cosimple bicircular matroids of corank at least two  that do not
have positive double circuits. Dually, $B(G)^\ast$
and $B(P)^\ast$ are simple matroids of rank at least two  that do not
have positive colines.


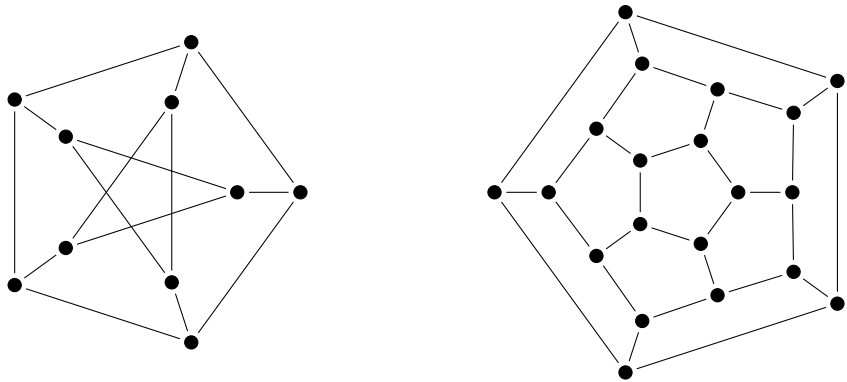
\begin{figure}[ht!]
\begin{center}

\begin{tikzpicture}[scale=0.6]

\begin{scope}[xshift=-6cm, yshift = 0cm, scale=0.7]
\node [vertex] (0) at (0:3){};
\node [vertex] (1) at (288:3){};
\node [vertex] (2) at (216:3){};
\node [vertex] (3) at (144:3){};
\node [vertex] (4) at (72:3){};

\node [vertex] (10) at (0:5){};
\node [vertex] (11) at (288:5){};
\node [vertex] (12) at (216:5){};
\node [vertex] (13) at (144:5){};
\node [vertex] (14) at (72:5){};

\foreach \from/\to in {0/2, 0/3, 1/3, 1/4, 2/4}
\draw [edge] (\from) to (\to);
\foreach \from/\to in {10/11, 12/13, 11/12, 13/14, 10/14}
\draw [edge] (\from) to (\to);
\foreach \from/\to in {0/10, 1/11, 2/12, 14/4, 3/13}
\draw [edge] (\from) to (\to);

\end{scope}

\begin{scope}[xshift=6cm, yshift=0cm, scale=0.6]
\node [vertex] (0) at (0:2){};
\node [vertex] (1) at (288:2){};
\node [vertex] (2) at (216:2){};
\node [vertex] (3) at (144:2){};
\node [vertex] (4) at (72:2){};

\node [vertex] (01) at (0:4){};
\node [vertex] (101) at (36:5){};
\node [vertex] (41) at (72:4){};
\node [vertex] (141) at (108:5){};
\node [vertex] (31) at (144:4){};
\node [vertex] (131) at (180:5){};
\node [vertex] (21) at (216:4){};
\node [vertex] (121) at (252:5){};
\node [vertex] (11a) at (288:4){};
\node [vertex] (111) at (324:5){};

\node [vertex] (10) at (36:7){};
\node [vertex] (11) at (324:7){};
\node [vertex] (12) at (252:7){};
\node [vertex] (13) at (180:7){};
\node [vertex] (14) at (108:7){};

\foreach \from/\to in {0/1, 2/3, 1/2, 3/4, 0/4}
\draw [edge] (\from) to (\to);
\foreach \from/\to in {10/11, 12/13, 11/12, 13/14, 10/14}
\draw [edge] (\from) to (\to);
\foreach \from/\to in {101/10, 111/11, 121/12, 141/14, 131/13}
\draw [edge] (\from) to (\to);
\foreach \from/\to in {01/0, 11a/1, 21/2, 41/4, 31/3}
\draw [edge] (\from) to (\to);

\foreach \from/\to in {01/101, 101/41, 41/141, 141/31, 31/131, 131/21, 21/121, 121/11a,
				11a/111, 111/01}
\draw [edge] (\from) to (\to);

\end{scope}

\end{tikzpicture}

\caption{The Petersen graph and the dodecahedron.}
\label{fig:petdod}
\end{center}
\end{figure}

Recall that every transversal matroid is a gammoid, and since bicircular matroids
are transversal matroids~\cite{matthewsQJMOS28}, Theorem~\ref{thm:main} 
yields a class of cosimple gammoids of corank at least two  that do not have
positive double circuits.

\begin{corollary}\label{cor:counter}
Not every cosimple  gammoid of corank at least two has a positive double
circuit. Equivalently, not every simple gammoid of rank at least two has a
positive coline.
\end{corollary}


\section{Conclusions}
\label{sec:conclusions}

A simple observation used to prove Theorem~\ref{thm:main} shows that the degree
of double circuits in bicircular matroids is bounded above by $6$. This raises
the natural problem of describing the classes of matroids obtained by considering
bicircular matroids whose double circuits have degree at most $k$ where
$k\in\{3,4,5\}$. This question has been answered from
another perspective for $k = 3$: the positive double circuits in a bicircular matroid
$B$ have degree at most $3$ if and only if $B$ is a binary bicircular matroid
\cite{matthewsQJMOS28}.
If we also restrict the degree of colines, for the case $k = 4$ we recover
ternary bicircular matroids \cite{sivaramanDM328}. For the case $k = 5$, considering
bicircular matroids with double circuits and colines  of degree at most $k$, we do
not recover representation over $GF(4)$ since $P_6$ has no positive double
circuits nor colines of degree greater that or equal to $5$, but is a bicircular matroid
not representable over $GF(4)$ \cite{chunDM339}. 
We are interested in knowing if there is a meaningful description of 
bicircular matroids whose double circuits have bounded degree.


\begin{problem}
Provide a meaningful description of bicircular matroids where every double
circuit has degree at most $k$ for  $k\in\{4,5\}$.
\end{problem}

The motivation of this work was to settle Conjecture~\ref{conj:coline};
we showed that it does not hold even for cobicircular matroids. 
Nonetheless, there are $GSP$ oriented matroid, whose underlying matroid
does not have a positive coline, for instance $P_7$ \cite{goddynDM339}. 
So, it still makes sense to ask the following question.

\begin{question}
Is every oriented cobicircular matroid $GSP$?
\end{question}

\section*{Acknowledgements}
This work was carried out during a visit of the first author
at FernUniversit\"at in Hagen, supported by DAAD grant 57552339.

\end{document}